\documentclass[12pt,letterpaper]{amsart}

\usepackage{fullpage}
\usepackage{pslatex} 
\usepackage{amsmath}
\usepackage{amssymb}
\usepackage{amsthm}
\usepackage{color}
\usepackage{array}
\usepackage{xy}
\usepackage{setspace}
\usepackage{multicol}
\usepackage{tikz}
\usepackage{comment}
\usepackage{hyperref}
\usepackage{cite}

\newtheoremstyle{slanted}
{0pt}
{0pt}
{\slshape}
{}
{\bfseries}
{.}
{.5em}
{}
\theoremstyle{slanted}
\newtheorem{theorem}{Theorem}[section]
\newtheorem{lemma}[theorem]{Lemma}

\newtheorem{corollary}[theorem]{Corollary}
\newtheorem{conjecture}[theorem]{Conjecture}

\newtheorem{definition}[theorem]{Definition}
\theoremstyle{remark}

\usepackage{enumitem}
\usepackage{todonotes}
\setlist[itemize]{leftmargin=*}
\setlist[enumerate]{leftmargin=*}

\begin{document}

\title[New Results on Pattern-Replacement Equivalences: Generalizing a Classical Theorem and Revising a Recent Conjecture]{New Results on Pattern-Replacement Equivalences: Generalizing a Classical Theorem and Revising a Recent Conjecture}
\date{February 25, 2018}
\author{Michael Ma}
\maketitle

\begin{abstract}

In this paper we study pattern-replacement equivalence relations on the set $S_n$ of permutations of length $n$. Each equivalence relation is determined by a set of patterns, and equivalent permutations are connected by pattern-replacements in a manner similar to that of the Knuth relation.

One of our main results generalizes the celebrated Erdos-Szekeres Theorem for permutation pattern-avoidance to a new result for permutation pattern-replacement. In particular, we show that under the $\{123 \cdots  k, k \cdots 321\}$-equivalence, all permutations in $S_n$ are equivalent up to parity when $n \ge \Omega(k^2)$.

Additionally, we extend the work of Kuszmaul and Zhou on an infinite family of pattern-replacement equivalences known as the rotational equivalences. Kuszmaul and Zhou proved that the rotational equivalences always yield either one or two nontrivial equivalence classes in $S_n$, and conjectured that the number of nontrivial classes depended only on the patterns involved in the rotational equivalence (rather than on $n$). We present a counterexample to their conjecture, and prove a new theorem fully classifying (for large $n$) when there is one nontrivial equivalence class and when there are two nontrivial equivalence classes.

Finally, we computationally analyze the pattern-replacement equivalences given by sets of pairs of patterns of length four. We then focus on three cases, in which the number of nontrivial equivalence classes matches an OEIS sequence. For two of these we present full proofs of the enumeration and for the third we suggest a potential future method of proof.

\end{abstract}

\newpage

\section{Introduction}

Over the past thirty years, permutation patterns have become one of
the most widely studied areas in enumerative combinatorics.

A permutation $w$ is said to contain a pattern $\pi$ if some
subsequence of $w$'s letters appear in the same relative order as do
the letters in $\pi$. For example, the permutation $w = 21354$
contains the pattern $\pi = 132$, since it contains several three
letter subsequences, including $2, 5, 4$, for which the first letter
is smallest, the second is largest, and the third is middle-valued.  

Permutation patterns first appeared in Donald Knuth's \emph{Art of
Computer Program} in 1968 \cite{knuth68}, in which Knuth characterized
the permutations avoiding the pattern $231$ as precisely those which
are stack sortable, and further showed that the number of such
permutations of length $n$ is the $n$-th Catalan number
$C_n$. Motivated by the elegance of Knuth's result, in 1985, Simion
and Schmidt initiated the systematic study of permutation pattern
avoidance \cite{simion85}. Since then permutation patterns have found
applications throughout combinatorics, as well as in computer science,
computational biology, and statistical mechanics \cite{kitaev11}.

In this paper, we study what are known as pattern-replacement
equivalences \cite{west2010equivalence, west2011adjacent, fazel2013equivalence, kuszmaul2013equivalence, kuszmaul2013counting, kuszmaul2014new}. Consider a set of patterns such as $\{123, 231\}$. Given
a permutation $w$ which contains a $123$ pattern, rearrange the
letters within that pattern so that they form a $231$ pattern. The
resulting permutation $w'$ is said to be equivalent to $w$ under
$\{123, 231\}$-equivalence. For example, the permutation $13524$
contains the pattern $123$ in the letters $1, 2, 4$; rearranging those
letters into the order $2, 4, 1$ which forms a $231$-pattern, we
obtain the permutation $23541$. Hence $13524$ is equivalent to
$23541$. More generally, we say that two permutations $a$ and $b$
are equivalent under the $\{123, 231\}$-equivalence if $b$ can be
reached from $a$ through a series of pattern-replacements in which
either a $123$ pattern is rearranged to a $231$ pattern, or a $231$
pattern is rearranged to a $123$ pattern. For example, $1342 \equiv
2431$ because rearranging $3, 4, 2$ to form a $123$ pattern transforms
$1342$ into $1234$, and then rearranging $1, 3, 4$ to form a $231$
pattern results in $2431$. The $\{123, 231\}$-equivalence forms an
equivalence relation over $S_n$, the set of length-$n$ permutations,
and thus partitions $S_n$ into equivalence classes.

The equivalence relation described above is special example of a more
general type of equivalence relation.

\begin{definition} Let $\Pi \subseteq S_c$ be a set of patterns. Two
permutations $a$ and $b$ in $S_n$ are \emph{equivalent} under the
\emph{$\Pi$-equivalence} if $a$ can be reached from $b$ through a series of
pattern-replacements in which one pattern from $\Pi$ is being replaced
with another. \emph{$\Pi$-equivalence} induces an equivalence relation on the permutations of $S_n$.
\end{definition}

One can also study a variant of the $\Pi$-equivalence in which adjacency constraints are imposed on the patterns. In order for a pattern to be eligible for rearrangement, all of its letters must appear adjacently. For example, although any three letters in the permutation $1234$ from a $123$ pattern, only two of those patterns, $123$ and $234$ satisfy the adjacency constraint. Formally, the \emph{$\Pi$-equivalence with adjacency constraints} is defined as follows.

\begin{definition} Let $\Pi \subseteq S_c$ be a set of patterns. Two
permutations $a$ and $b$ in $S_n$ are \emph{equivalent} under the
\emph{$\Pi$-equivalence with adjacency constraints} if $a$ can be reached from $b$ through a series of pattern-replacements in which one pattern from $\Pi$ is being replaced with another, and in which the letters in the pattern being rearranged appear adjacently in $a$.
\end{definition}

The earliest studied permutation pattern-replacement equivalences are the Knuth equivalence \cite{knuth1970permutations} and the Forgotten equivalence \cite{novelli2007forgotten}. The Knuth relation, also known as the plactic equivalence, has found applications in both combinatorics and algebra. It played an important role in one of the first proofs of the Littlewood-Richardson rule, a rule which can be used as a identity for multiplying Schur polynomials. The Forgotten equivalence has also found numerous applications in abstract algebra \cite{van1996tableau, novelli2007forgotten}.

The first systematic study of permutation pattern-replacement equivalences was initiated by \cite{west2010equivalence} and \cite{west2011adjacent} who made substantial progress characterizing the equivalence classes for pattern-replacement equivalences involving patterns of length three. Subsequent research \cite{kuszmaul2014new, kuszmaul2013counting} continued to focus on the case of patterns of length $3$. Then, in 2013, \cite{kuszmaul2013equivalence} presented the first results on infinite families of permutation patterns. They showed that for any set $\Pi$ consisting of some pattern and its cyclic shifts, the number of nontrivial equivalence classes in $S_n$ under $\Pi$-equivalence with adjacency constraints is at most two. Note that a nontrivial equivalence class is one containing more than a single permutation.

In this paper, we continue the study of infinite families of equivalence relations, and initiate the systematic study of pattern-replacement equivalences involving patterns of length four. Our main result is a generalization of the celebrated Erd\"{o}s-Szekeres Theorem \cite{erdos1935combinatorial} for permutation pattern avoidance to the setting of pattern-replacement equivalences. We show that the $\{123\cdots k, k \cdots 321\}$-equivalence partitions $S_n$ into at most two equivalence classes for $k \in \Omega(n^2)$. In particular, there is a single equivalence class when $k \pmod 4 \in \{2, 3\}$, and there are two equivalence classes determined by permutation parity otherwise. Note that this implies a weak version of the traditional Erd\"{o}s-Szekeres Theorem, since it prohibits sufficiently long permutations from simultaneously avoiding the $123 \cdots k$ and $k \cdots 321$ patterns.

Our second main result extends the work of \cite{kuszmaul2013counting} on rotational equivalences. Kuszmaul and Zhou proved that the rotational equivalences always yield either one or two nontrivial equivalence classes in $S_n$, and conjectured that the number of nontrivial classes depended only on the patterns involved in the rotational equivalence (rather than on $n$) \cite{kuszmaul2013equivalence}. We present a counterexample to their conjecture, and prove a new theorem fully classifying (for large $n$) when there is one nontrivial equivalence class and when there are two nontrivial equivalence classes.


Finally, we collect computational data on the number of nontrivial equivalences classes induced by the $\Pi$-equivalence for $\Pi$ containing two patterns of length four. In order to identify particularly interesting selections of $\Pi$, we focus on three cases where the number of nontrivial equivalence classes induced by $\Pi$-equivalence appears to be enumerated by an OEIS sequence. For two of these cases, we are able to prove the enumeration, and we pose the third as a conjecture.


The rest of the paper proceeds as follows. In Section \ref{secerdos} we present our generalization of the Erd{\"o}s-Szekeres Theorem. In Section \ref{secrot} we disprove a conjecture of \cite{kuszmaul2013equivalence} and prove a refined version of the conjecture. In Section \ref{secoeis} we consider three OEIS sequences dictated by pairs of patterns of length $4$. Finally, we concluded in Section \ref{secconclusion} with directions of future work.

\section{Generalizing Erd{\"o}s-Szekeres}
\label{secerdos}
The Erd{\"o}s-Szekeres Theorem is one of the oldest and most celebrated results in the area of permutation pattern avoidance \cite{erdos1935combinatorial}. The result was originally introduced in order to obtain an alternative proof of Ramsey's Theorem, and has since been cited more than 1,300 times \cite{erdos1935combinatorial}. Previous research on permutation pattern replacement equivalences has yielded several relationships between permutation pattern avoidance and permutation pattern replacement; in particular, several cases are known where the set of permutations avoiding a set of patterns is also a set of representatives for the equivalence classes under a particular equivalence relation \cite{kuszmaul2013counting}. In this section, we present a generalization of the Erd\"{o}s-Szekeres Theorem for permutation pattern-replacement equivalences, further strengthening the connection between the two areas of research.


The classical Erd\"{o}s-Szekeres Theorem can be formulated as follows: Within any sequence of $(r-1)(k-1)+1$ distinct numbers there always exists either an increasing subsequence of length $r$ or a decreasing subsequence of length $k$. If we let $r=k$ then it becomes the following: Within any sequence of $(k-1)^2+1$ distinct numbers there always exists either an increasing subsequence or a decreasing subsequence of length $k$. In terms of pattern avoidance, this means no permutation of length greater than $(k-1)^2 + 1$ can avoid both $12 \cdots k$ and $k \cdots 21$. The natural generalization to pattern replacement equivalences arises from studying the number of equivalence classes of $S_n$ under $\{12 \cdots k, k \cdots 21\}$-equivalence when $n > (k-1)^2$. The Erd\"{o}s-Szekeres Theorem tells us that there are no singleton equivalences, meaning that every permutation is in a nontrivial class. A natural generalization of the theorem would be to prove that there is only a single equivalence class, which in turn would also imply the original result. It turns out that when $k$ is zero or one modulo four, parity is an invariant under $\{12 \cdots k, k \cdots 21\}$-equivalence, meaning that odd and even permutations cannot be equivalent. Thus the most natural generalization of the Erd\"{o}s-Szekeres Theorem to pattern replacement would be to prove that there is a single equivalence class when $k \pmod 4 \in \{2, 3\}$ and that there are two classes consisting of the even permutations and odd permutations when $k \pmod 4 \in \{0, 1\}$. This, in turn, would imply the original theorem since it would prohibit the existence of avoiders. 

One might hope that the generalization would hold for all $n \ge (k - 1)^2 + 1$, just as does Erd\"{o}s-Szekeres. One can verify, however, that for $k = 3$ and $n = 5$, the generalization does not hold, and there are instead three equivalence classes. Nonetheless, we will prove the generalization for sufficiently large $n \in \Omega(k^2)$. Specifically, in this section we prove that for $n \ge 3k^2-4k+3$ the number of equivalence classes in $S_n$ under $\{12 \cdots k, k \cdots 21\}$-equivalence is always either one or two, depending on whether parity is an invariant.


As a convention, for each equivalence class $A \subseteq S_n$ under an equivalence relation, we will use $p_A$ to denote the lexicographically smallest element in the equivalence class. Our first lemma proves a property of $p_A$ under $\{12 \cdots k, k \cdots 21\}$-equivalence. 

\begin{lemma}
Let $p_A = a_1a_2 \cdots a_n$ be the lexicographically smallest element in some equivalence class $A \subseteq S_n$ under $\{12 \cdots k, k \cdots 21\}$-equivalence. Then for any $i$ between $1$ and $n$ we have $i-(k-1)^2 \leq a_i \leq i+(k-1)^2.$
\label{lempAfirstbound}

\end{lemma}

\begin{proof}

Assume for the sake of contradiction that $a_i < i-(k-1)^2$. Then look at all the letters appearing prior to $a_i$ in the permutation $p_A$. There are $i-1$ such letters and at most $i-(k-1)^2-2$ of them have value less than $a_i$. Therefore at least $(k-1)^2+1$ of the letters preceding $a_i$ are greater in value than $a_i$. By Erd{\"o}s-Szekeres, we know that there is either an increasing or decreasing sub-sequence of length $k$ within these letters. Suppose the subsequence is given by $a_{i_1},a_{i_2}, \ldots, a_{i_k}$. Because $p_A$ is lexicographically minimal within the class $A$, the subsequence of length $k$ must be increasing. Using the equivalence relation, we can first reverse the sequence to be in the order $a_{i_k}, \ldots, a_{i_1}$. Because each of $a_{i_1}, \ldots, a_{i_k}$ appear prior to $a_i$ and are of value greater than $a_i$, in the new permutation, the subsequence $a_{i_k},\ldots ,a_{i_3},a_{i_2},a_i$ is in decreasing order, and can thus be reversed to place $a_i$ in the position originally occupied by $a_{i_1}$. Because $a_i < a_{i_1}$, it follows that $p_A$ is equivalent to a permutation which is lexicographically smaller than it, a contradiction. By a symmetric argument we can get the other side of the inequality by considering the letters appearing after $a_i$ and numerically smaller in value.
\end{proof}

\begin{corollary}

Let $p_A = a_1a_2 \cdots a_n$ be defined as in the preceding lemma. Then $i$ appears in $p_A$ within the subword $a_{i-(k-1)^2} \cdots  a_{i+(k-1)^2}$.
\label{corpAfirstbound}
\end{corollary}

\begin{proof}
This follows from the fact that Lemma \ref{lempAfirstbound} prohibits $i$ from appearing anywhere outside the subword $a_{i-(k-1)^2} \cdots  a_{i+(k-1)^2}$.
\end{proof}

\noindent We can now use the preceding lemma and corollary to prove another, more powerful lemma.

\begin{lemma}

If $n \geq 3k^2-6k+6$ then each permutation in $S_n$ is equivalent to some permutation beginning with $1$ under the $\{12 \cdots k,k \cdots 21\}$-equivalence. 
\label{lembeginone}
\end{lemma}

\begin{proof}

Take a permutation in $S_n$. Within its equivalence class, $A$, look at $p_A = a_1a_2 \cdots a_n$ as defined previously. Assume for the sake of contradiction $p_A$ doesn't begin with a $1$. Then by Lemma \ref{lempAfirstbound} the first letter in the permutation must be at most $(k-1)^2+1$ in value. By Corollary \ref{corpAfirstbound} we know that there are at most $(k-1)^2$ letters in the permutation to the right of $n$. So the number of letters which are strictly between $a_1$ and $n$ in both value and position is at least \[n-2-2(k-1)^2 \geq k^2-2k+2 = (k-1)^2+1\]. Thus we can find either an increasing or decreasing sequence of $k$ letters within the letters which are between $a_1$ and $n$ in both position and value. Since $p_A$ is the lexicographically smallest representative in the class $A$, we know this subsequence is increasing. Denote the first $k - 2$ letters in the increasing subsequence by $a_{i_1}a_{i_2}, \ldots , a_{i_{k-2}}$. Then use the equivalence relation to reverse the sub-sequence $a_1,a_{i_1},a_{i_2}, \ldots, a_{i_{k-2}},n$, giving a new permutation $p_A'$. Notice that $n$ is the first letter in $p_A'$. Also the number $1$ is in the same position as in $p_A$ since it is less than $a_1$ and thus cannot have been any of the $a_{i_l}$'s. 

By Corollary \ref{corpAfirstbound} in our new permutation $p_A'$ we know there are at most $(k-1)^2$ letters to the left of $1$. Thus there are least \[n-(k-1)^2-1 \geq 2k^2-4k+4 > k^2-2k+2 = (k-1)^2+1\] letters to the right of $1$ in $p_A'$, meaning by Erd\"{o}s-Szekeres that the letters to the right of one contain either an increasing or a decreasing subsequence of $k$ letters. We can make such a subsequence increasing by reversing if decreasing and then we can denote the letters in this subsequence by $a_{j_1},a_{j_2}, \ldots, a_{j_{k-1}}$. Using the equivalence relation, we can change the sub-sequence $n1a_{j_1}a_{j_2} \cdots a_{j_{k-1}}$ to $na_{j_{k-1}}a_{j_{k-2}} \cdots a_{j_1}1$, and then to $1a_{j_{k-1}}a_{j_1}a_{j_2} \cdots a_{j_{k-2}}n$ which gives us an equivalent permutation to $p_A$ that begins with $1$, a contradiction. Therefore, $p_A$ begins with $1$.
\end{proof}

\begin{corollary}

Under the $\{12 \cdots k, k \cdots 21\}$-equivalence, each permutation is equivalent to one that matches the identity for the first $n-(3k^2-6k+5)$ letters.
\label{corbeginid}
\end{corollary}

\begin{proof}

This follows from repeated applications of Lemma \ref{lembeginone}.
\end{proof}

The preceding corollary shows that every equivalence class contains a permutation with a long prefix matching the identity permutation. The next lemma extends this to show that there are always at most two equivalence classes. Moreover, the lemma keeps track of parity in a way that will prove useful when characterizing the equivalence classes.

\begin{lemma}

Let $k \ge 3$ and let $n \ge 3k^2 - 4k + 3$. Then every permutation $a \in S_n$ beginning with $12 \cdots (2k - 2)$ is equivalent to either $12 \cdots n$ or $213 \cdots n$ via an even number of pattern-replacements under the $\{12\cdots k, k \cdots 21\}$-equivalence.
\label{lembecomeid}
\end{lemma}

\begin{proof}

Consider a non-identity permutation $a = a_1 \cdots a_n \in S_n$ beginning with $12 \cdots (2k - 2)$.  Then let $l$ be the largest position such that $a_i=i$ for all $1 \leq i \leq l$. Since $a$ is not the identity, we know $l \leq n-2$. Note that $a_{l+1}>l+1$. Take the subsequence $1, 2, \ldots, (2k-2), a_{l+1}, (l+1)$. We can repeatedly apply pattern replacements to the subsequence, as follows: (the overlined portion of the subsequence in each step forms the pattern being rearranged.) 

\begin{align*}
    & 1, 2, 3, 4, \ldots, (2k - 2), a_{l + 1}, (l + 1) \\
= \ & \overline{1},2,\overline{3,4, \ldots, k}, (k+1),(k+2), \ldots, (2k-2), \overline{\vphantom{+1} a_{l+1}}, (l+1) \\
\rightarrow \ & \overline{\vphantom{+1} a_{l + 1}},2,\overline{k,(k - 1), \ldots, 3}, (k+1),(k+2), \ldots, (2k-2), \overline{1}, (l+1) \\ 
= \ & a_{l+1}, \overline{2}, k,(k-1), \ldots, 3, \overline{(k+1),(k+2), \ldots, (2k-2)}, 1, \overline{(l+1)} \\
\rightarrow \ & a_{l+1}, \overline{(l + 1)}, k,(k-1), \ldots, 3, \overline{(2k - 2), (2k - 3), \ldots, (k + 1)}, 1, \overline{2} \\
= \ & \overline{\vphantom{+1} a_{l+1}}, (l+1), \overline{k,(k-1), \ldots, 3}, (2k-2),(2k-3), \ldots, (k+1), 1, \overline{2} \\
\rightarrow \ & \overline{2}, (l+1), \overline{3,4, \ldots, k}, (2k-2),(2k-3), \ldots, (k+1), 1, \overline{\vphantom{+1} a_{l + 1}} \\
= \ & 2, \overline{(l+1)}, 3,4, \ldots, k, \overline{(2k-2),(2k-3), \ldots, (k+1),1}, a_{l+1} \\
\rightarrow \ & 2, \overline{1}, 3,4, \ldots, k, \overline{(k + 1), (k + 2), \ldots, (2k - 2),(l + 1)}, a_{l+1} \\
= \ & 2, 1, 3, 4, \ldots, (2k-2), (l+1), a_{l+1}.
\end{align*}



If the new permutation is $2134 \cdots n$, then we are done, since we applied an even number of transformations to $a$. Otherwise, the new permutation looks like $2134 \cdots m a_{m+1} \cdots a_{n}$ where $a_{m+1} \neq m+1$ and $l+1 \leq m \leq n-2$. Then by repeating the above process with the roles of $1$ and $2$ swapped we end up with a permutation that begins with $12 \cdots (m+1)$, having applied a total of $8$ transformations to $a$. Repeating the entire process on the new permutation, we will eventually reach either the identity or the identity with $1$ and $2$ swapped. Since each time we use the process, we apply an even number of transformations, in total we will have applied an even number of transformations. Therefore, the parity of the final permutation will be the same as that of the initial.
\end{proof}

\noindent We now prove our main result, which generalizes Erd{\"o}s-Szekeres.

\begin{theorem}

Let $k \ge 3$. Then if $n \geq 3k^2-4k+3$ then the number equivalence classes in $S_n$ under the $\{12 \cdots k, k \cdots 21\}$-equivalence is one when $k \pmod 4 \in \{2, 3\}$ and is two when $k \pmod 4 \in \{0, 1\}$.
\label{thmgeneralize}
\end{theorem}

\begin{proof}
Let $n \geq 3k^2-4k+3$. By Corollary \ref{corbeginid}, every permutation is equivalent to a permutation beginning with $12 \cdots (2k - 2)$. By Lemma \ref{lembecomeid}, all permutations beginning with $12 \cdots (2k - 2)$ are equivalent to one of $1234 \cdots n$ or $2134 \cdots n$ via an even number of transformations. Now if $k \pmod 4 \in \{0, 1\}$, then we know that the equivalence does not allow even and odd permutations to be equivalent since one replacement is equivalent to $\lfloor \frac{k}{2} \rfloor$ transpositions. In this case we have our two equivalence classes. Suppose instead that $k \pmod 4 \in \{2, 3\}$. We have shown that all even permutations beginning with $123 \cdots (2k-2) \cdots$ are equivalent to $123 \cdots n$ and that all odd permutations beginning with $123 \cdots (2k-2) \cdots$ are equivalent to $2134\cdots n$. But we know that \[123 \cdots n \equiv 123 \cdots (n-k)n(n-1) \cdots (n-k+1)\] which is a pair of even and odd permutations beginning with $123 \cdots (2k-2) \cdots$. Thus all permutations beginning with $123 \cdots (2k - 2)$ are equivalent, implying that all permutations are equivalent.
\end{proof}

\noindent We conjecture that Theorem \ref{thmgeneralize} can be strengthened further to the following. The conjecture has been experimentally verified for $k \le 4$.

\begin{conjecture}

Let $k \ge 3$. Then if $n \geq k^2-2k+3$ there are only one or two equivalence classes under the equivalence $\{12 \cdots k, k \cdots 21\}$. Furthermore if $k$ is even we need only that $n \geq k^2-2k+2$.

\end{conjecture}

\section{The Rotational Equivalence}\label{secrot}

\subsection{Introduction}

The first infinite family of pattern replacement equivalences to be studied was rotational equivalences by  \cite{kuszmaul2013equivalence}. In the paper, it was proven that for any pattern $m$, if $\Pi$ is the set of cyclic rotations of the pattern, then there are always at most two nontrivial equivalence classes under $\Pi$-equivalence with adjacency constraints (meaning that pattern-replacements may only be performed on adjacent blocks of letters). This equivalence relation is known as the \emph{$m$-rotational equivalence}.

Furthermore Conjecture 2.12 in \cite{kuszmaul2013equivalence} states that for a given pattern $m$, the number of nontrivial classes in $S_n$ (for $n$ larger than $|m|$) is independent of $n$. In this section, we present a counterexample for the conjecture, and prove a revised version of the conjecture. 

When $m$ has an odd length there are are always two nontrivial classes due to a parity invariant. Therefore, for the remainder of the section we will focus only on cases where $|m|$ is even, and for convenience $|m|$ is denoted by $c$.



The conjecture stated in \cite{kuszmaul2013equivalence} can be disproven by computational brute force. For example, when $m = 1324$ and $n=6$ we find that there are $2$ nontrivial equivalence classes, but when $n=7$ there is only $1$ nontrivial equivalence class. More generally, computational data motivates the following result:


\begin{theorem}

Let $m$ be a pattern of even length $c$. Let $f(n)$ denote the number of nontrivial equivalence classes in $S_n$ under $m$-rotational equivalence. Then there is some cutoff $t \le 2c - 1$ such that
\[
f(n) = \begin{cases}
2 & \text{if}  \ \ \ c < n < t\\
1 & \text{if} \ \ \  n \ge t
\end{cases}
\]
\label{conjmain}
\end{theorem}

A striking feature of Theorem \ref{conjmain} is that for all $n \ge 2c - 1$, there is only a single equivalence class under $m$-rotational equivalence, regardless of the choice of $m$. Furthermore, we conjecture a slightly stronger version of the theorem, which is that $t = 2c - 1$ if and only if $m$ is alternating. 

In the following two subsections, we present a proof of Theorem \ref{conjmain}. Subsection \ref{subsecneweq} defines a notion of a pseudo-permutation and constructs an equivalence relation on pseudo-permutations which mimics rotational equivalence; by characterizing the equivalence classes for pseudo-permutations, we are able to use them as a lens through which to study rotational-equivalence classes for permutations. Subsection \ref{subsecfinal} then uses our results on pseudo-permutations in order to prove Theorem \ref{conjmain} for permutations. 

\subsection{A New Equivalence}\label{subsecneweq}

In this section we present a notion of pseudo-permutations and define a useful equivalence relation on them which mimics $m$-rotational equivalence on permutations. For the rest of the section, for simplicity, we will assume that $m$ begins with a $1$ (which by rotational symmetry may be taken without loss of generality).



Given a permutation $w \in S_n$ with a pattern in a particular position, it can be useful to abstract away the pattern by considering what we call a pseudo-permutation. The pseudo-permutation is obtained by replacing the pattern with a single letter denoted by $p$. Note that $p$ does not take an integer value, and can be thought of instead as being a formal variable representing a pattern. This is formalized in the following two definitions.

\begin{definition}
A \textbf{pseudo-permutation} of length $n$ is a word of length $n - c + 1$ letters consisting of the letter $p$ and of $n-c$ distinct letters from $1$ to $n$.

\end{definition}

\begin{definition}
Given a pseudo-permutation $\tau$, let $\alpha$ be the unique $c$-letter word that is both order-isomorphic to $m$ and contains each of the letters from $[n]$ that are not in $\tau$. The \textbf{representative permutation} $\sigma$ of a pseudo-permutation $\tau$ is the permutation formed by replacing the letter $p$ within $\tau$ by the sub-sequence $\alpha$.
\end{definition}

For example, if $m = 1324$ and $\tau = 2p468$, then the representative permutation for $\sigma$ is $21537468$. In particular, $p$ expands to become the pattern $m$ formed by the letters $1, 3, 5, 7$. As a convention, we refer to the letters in $[n]$ which do not appear in $\tau$ as \emph{appearing within the pattern} in the pseudo-permutation.

For convenience we now introduce notation for identifying specific letters of interest within a pattern. 
\begin{definition}

Take a pseudo-permutation $\tau$ with representative permutation $\sigma$. For a letter $a \in [n]$, define $a^+$ as the smallest letter in $\tau$'s pattern that is larger than $a$. Similarly define $a^-$ as the largest letter within $\tau$'s pattern that is smaller than $a$. If no such letter exists, then we do not consider $a^+$ (or $a^-$, respectively) to be well defined.

\end{definition}

We will proceed by defining an equivalence relation on pseudo-permutations which mimics the rotation-equivalence, while considering only a single pattern within the permutation.

\begin{definition}

Take a pseudo-permutation $\tau$. An \textbf{pseudo-rotation} on $\tau$ is any replacement in which either something of the form $pa$ within $\tau$ is replaced with $(a^{\pm})p$, or something of the form $ap$ within $\tau$ is replaced with $p(a^{\pm})$. Here $\pm$ denotes either $+$ or $-$, and we require that $a$ and $a^{\pm}$ appear adjacently to $p$ within $\tau$.

\end{definition}

It is important to notice that if a pseudo-rotation changes a pseudo-permutation $\tau_{1}$ to $\tau_{2}$ then there is another pseudo-rotation that instead changes $\tau_{2}$ to $\tau_{1}$. That is, the inverse of a pseudo-rotation is a pseudo-rotation.

\begin{definition}

We can say two pseudo-permutations, $\tau_{1}$ and $\tau_{2}$, are \textbf{equivalent} if and only if one can be reached from the other through a series of pseudo-rotations.

\end{definition}

Notice that the above definition induces an equivalence relation on pseudo-permutations. Moreover, as formalized by the following lemma, equivalence for pseudo-permutations can be viewed as a weak version of $m$-rotational equivalence.

\begin{lemma}

If two pseudo-permutations $\tau_{1}$ and $\tau_{2}$ are equivalent then their representative permutations $\sigma_{1}$ and $\sigma_{2}$ are equivalent under $m$-rotational equivalence.
\label{lemequiv}
\end{lemma}

\begin{proof}
For $i \in [c + 1]$, define $(i)m$ to be the unique permutation in $S_{c + 1}$ that begins with $i$ and whose final $c$ letters form the pattern $m$. Similarly, define $m(i)$ to be the unique permutation in $S_{c + 1}$ that ends with $i$ and whose first $c$ letters form the pattern $m$. Then, Lemma 2.7 of \cite{kuszmaul2013equivalence} states that for $j \in [c]$, we have that $(j)m$ is equivalent to $m(j + 1)$, and that $m(j)$ is equivalent to $(j + 1)m$ under $m$-rotational equivalence.

This implies that if a pseudo-rotation takes a pseudo-permutation $\tau_{a_1}$ to $\tau_{a_2}$, then their representative permutations $\sigma_{a_1}$ and $\sigma_{a_2}$ are equivalent. Repeatedly using this fact proves the desired result.

\end{proof}

In the remainder of this section, we will show that for $n \ge c + 1$, there are exactly two equivalence classes of pseudo-permutations of length $n$. Our next definition introduces a surprising parity invariant under pseudo-permutation equivalence. We will see that this invariant determines the equivalence classes.

\begin{definition}

Give a pseudo-permutation $\tau$ the \textbf{pseudo-parity} of $\tau$ is the sum modulo two of the number of inversions within $\tau$ (ignoring the letter $p$), the number of odd valued letters before $p$, and the number of even valued letters after $p$. If this is $0$, we will say $\tau$ is \textbf{even} and otherwise $\tau$ is \textbf{odd}.

\end{definition}

For example, if $\tau = 16p28$, then the pseudo-parity is $\operatorname*{inv}(1628) + 1 + 2 \mod 2$, which simplifies to $1 + 1 + 2 = 0 \mod 2$. Thus $\tau$ is said to be even in this case.

The next lemma establishes that pseudo-parity is an invariant.

\begin{lemma}

If pseudo-rotation takes $\tau_{1}$ to $\tau_{2}$, then the pseudo-parity does not change.

\end{lemma}

\begin{proof}

For simplicity, we will consider only the case where the pseudo-rotation changes $pa$ within $\tau_{1}$ to $a^+p$ within $\tau_{2}$ -- the other cases follow similarly. Define $k = a^+-a$. Now look at the difference in the pseudo-parities of $\tau_{1}$ and $\tau_{2}$. All inversions or non-inversions other than those involving $a$ or $a^+$ remain unchanged. The only change in the inversions is due to letters with values between $a$ and $a^+$ interacting with $a$ and $a^+$. Namely, we must consider the letters \[a+1, a+2, \cdots, a+k-1,\] each of which is either involved in one more or one fewer inversions in $\tau_2$ than in $\tau_1$. So the change in the number of inversions is the same parity as $k-1$. Now, as for the changes in the number (modulo two) of odd valued letters before $p$ and even valued letters after $p$, there are two cases: If $k$ is odd then there the sum modulo two is unchanged due to $a$ and $a^+$ having opposite parities; and if $k$ is even then the sum modulo two changes since $a$ and $a^+$ are the same parity and appear on opposite sides of $p$. So once again the change has the same parity as $k-1$. Hence the overall pseudo-parity remains invariant.

\end{proof}

The next lemma establishes that pseudo-equivalence is determined entirely by pseudo-parity when $n = c + 1$ or $n = c + 2$. We will then use the lemma in order to prove the same result for all $n \ge c + 1$.
\begin{lemma}

If $n = c + 1$ or $n = c + 2$ then there are two equivalence classes on pseudo-permutations of length $n$.
\label{thmbasecase}
\end{lemma}

\begin{proof}

Before beginning, recall that $c$ is even by assumption, a fact which we will use repeatedly.

Because pseudo-parity is an invariant $123 \cdots (n-c)p$ and $213 \cdots (n-c)p$ have different pseudo-parities, meaning there must be at least two equivalence classes. Now we wish to show that all pseudo-permutations of a given parity are equivalent.
First for $n = c + 1$ we can see that the pseudo-permutations $1p, p2, 3p, \ldots, pc, (c+1)p$ are equivalent and that the pseudo-permutations $p1, 2p, p3, \ldots, cp, p(c+1)$ are equivalent, completing the proof in that case. 

For $n = c + 2$ we can manually verify the theorem as follows. First, notice that any pseudo-permutation is equivalent to another pseudo-permutation for which $p$ is not the first or last letter. Moreover, by the already proven case of $c + 1$, we know that for each $i \neq p$, all of the even pseudo-permutations beginning with $i$ are equivalent. Therefore, in order to establish that all even pseudo-permutations are equivalent, it suffices to show that for all $i, j \neq p$ there is an even pseudo-permutation beginning with $i$ which is equivalent to an even pseudo-permutation beginning with $j$. When $i$ and $j$ are odd, this follows from the sequence,
\[1p(c+2) \equiv 3p(c+2) \equiv \cdots \equiv (c-1)p(c+2) \equiv (c+1)p(c+2).\]
Moreover, the even pseudo-permutations beginning with $1$ are equivalent to those beginning with $2$ because 
\[1p2 \equiv 4p2 \equiv 4p5 \equiv 2p5 \equiv 2p(c+1).\]
The even pseudo-permutations beginning with $2$, are in turn equivalent to those beginning with any even $i \neq c+ 2$ because
\[2p(c+1) \equiv 4p(c+1) \equiv \cdots \equiv (c-2)p(c+1) \equiv cp(c+1).\]
And finally, the even pseudo-permutations beginning with $1$ are equivalent to those beginning with $(c + 2)$, since
\[1p2 \equiv 4p2 \equiv (c+2)p2 \equiv (c+2)pc.\]
Combining these, we see that the even pseudo-permutations form a single equivalence class. A similar analysis of the odd pseudo-permutations establishes that there are two equivalence classes in the case of $n = c + 2$.

\end{proof}

We conclude the subsection by generalizing the preceding lemma to hold for all $n \ge c + 1$, thereby establishing pseudo-parity as a complete invariant.

\begin{theorem}

If $n \ge c+1$ then there are two equivalence classes on pseudo-permutations of length $n$.
\label{thmpseudo}
\end{theorem}

\begin{proof}

By Lemma \ref{thmbasecase}, we may assume without loss of generality that $n \ge c + 3$. 

We introduce the notion of the \textbf{control block} of a pseudo-permutation, which is defined to be the letter $p$ along with the letter following it. If $p$ is the last letter then the control block is $p$ along with the letter before it. We will denote the control block by $C$, representing $p$ and the additional letter. For example, the pseudo-permutation $4p17$ has control block $p1$ and can be represented as $4C7$.

The key observation is that if $a$ is a letter and $C$ is a control block which contains the letter $b$, then psuedo-rotations can be used to transform $aC$ into $Ca$ and to transform $aC$ into $bC$, where the new control blocks contain different letters than the original.This follows from Lemma \ref{thmbasecase}, since the control blocks in $bC$ and in $Ca$ can be chosen to ensure that they are the same psuedo-parity as $aC$. (In particular, the fact that the control block is $c + 1$ letters allows us to use the $n = c + 1$ case from Lemma \ref{thmbasecase} in order to select a control block with either pseudo-parity.)

Armed with this we can finish the proof. Given a pseudo-permutation $\tau$, first move $\tau$'s control block so that it follows the letter $1$ (using operations of the form $aC \leftrightarrow Ca$); then use an operation of the form $C1 \rightarrow bC$ for some $b$ to move $1$ into the control block; then use operations of the form $aC \leftrightarrow Ca$ to move the control block to the second second position in the psuedo-permutation; and finally use an operation of the form $aC \rightarrow 1C$ to move $1$ into the first position. Continue like this to place $2$ in the second position, and so on, until we get a psuedo-permutation of the form $12\cdots (n - c - 1) C$. By the case of $n = c + 1$ applied to the control block $C$, there are only two equivalence classes containing such pseudo-permutations, completing the proof that there are at most two equivalence classes of pseudo-permutations of length $n$. Because pseudo-parity is an invariant, it follows that there are exactly two classes. 

\end{proof}

\subsection{Analyzing $m$-Rotational Equivalence}\label{subsecfinal}

In the previous section, we introduced an equivalence relation on pseudo-permutations which mimicked rotational equivalence. We were then able to fully characterize the equivalence classes under the equivalence relation for pseudo-permutations. Now, with the help of Theorem \ref{thmpseudo}, we can prove our main result, Theorem \ref{conjmain}. We restate the result below for convenience.

\begin{theorem}[Theorem \ref{conjmain} restated]

Let $m$ be a pattern of even length $c$. Let $f(n)$ denote the number of nontrivial equivalence classes in $S_n$ under $m$-rotational equivalence. Then there is some cutoff $t \le 2c - 1$ such that
\[
f(n) = \begin{cases}
2 & \text{if}  \ \ \ c < n < t\\
1 & \text{if} \ \ \  n \ge t
\end{cases}
\]

\end{theorem}

\begin{proof}

Consider permutations of length $n$ for some $n > c$. Notice that by Theorem \ref{thmpseudo} and Lemma \ref{lemequiv} we know there are at most two equivalence classes under $m$-rotational equivalence. Also if there exists a $t$ such that when $n = t$ there is one equivalence class, this means there are two pseudo-permutations of different pseudo-parity such that their representative permutations are equivalent. Now notice if we add $t+1$ to the end of both pseudo-permutations they still have different pseudo-parities and their representative permutations are still equivalent. This means by induction that for all $n \ge t$, there must be a single equivalence class in $S_n$. Thus, in order to complete the proof, it suffices to show when $n = 2c-1$ that there is one equivalence class.

Define $\sigma_1 = a_1a_2 \cdots a_{2c - 1}$ to be the permutation in $S_{2c - 1}$ such that $a_1 \cdots a_c$ are the letters $1 \cdots c$ rearranged to form the pattern $m$, and such that $a_{c + 1} \cdots a_{2c - 1}$ are the letters $(c + 1) \cdots (2c - 1)$ rearranged to form the final $c - 1$ letters of the pattern $m$. Recall that without loss of generality we assume $m$ begins with $1$. Hence the final $c$ letters $a_c \cdots a_{2c - 1}$ of $\sigma_1$ form the pattern $m$. As a result, $\sigma_1$ is the representative permutation for the pseudo-permutation,
$$\tau_1 = a_1 \cdots a_{c - 1} p.$$

On the other hand, via an $m$-rotation, $\sigma_1$ is equivalent to the permutation,
$$\sigma_2 = a_2 a_3 \cdots a_{c - 1} a_c a_1 a_{c + 1} a_{c + 2} \cdots a_{2c - 1}.$$
Moreover, because $a_1 = 1$ and the pattern $m$ begins with one, the final $c$ letters of $\sigma_2$ form an instance of the pattern $m$. This means that $\sigma_2$ is the representative permutation for the pseudo-permutation,
$$\tau_2 = a_2 \cdots a_{c} p.$$

We will show that $\tau_1$ and $\tau_2$ have different pseudo-parities. Since their representative permutations are equivalent, this implies that there is a single equivalence class in $S_{2c - 1}$ under $m$-rotational equivalence.

Since $a_1 = 1$, we have that $a_{1}a_{2} \cdots a_{c}$ and $a_{2} \cdots a_{c}$ have the same parity of number of inversions. Moreover, if we remove $a_c$ from $a_{1}a_{2} \cdots a_{c}$ then we change the parity by $a_c \mod 2$. So the difference in parity of inversions between $a_{1}a_{2} \cdots a_{c-1}$ and $a_{2}a_{3} \cdots a_{c}$ is $a_{c} \mod 2$. Moreover, notice that the difference (modulo two) in the number of odd valued letters before $p$ within $\tau_{1}$ and $\tau_{2}$ is $a_{c}+a_{1} \equiv a_c + 1 \mod 2$. Thus the pseudo-parities of $\tau_1$ and $\tau_2$ differ by $a_c + (a_c + 1) \equiv 1 \mod 2$, as desired.

\end{proof}

\section{Pattern-Replacement with Patterns of Length Four}\label{secoeis}

In past work, extensive effort has been made to study the special case of pattern-replacement equivalences involving patterns of length three \cite{kuszmaul2013equivalence, kuszmaul2014new, kuszmaul2013counting, west2010equivalence, west2011adjacent, novelli2018noncommutative}. The resulting enumerations have yielded many beautiful number sequences, Catalan numbers, sums of binomial coefficients, sums of Motzkin numbers, etc. The difficulty when studying equivalences with patterns of length four is selecting which equivalence relations to focus on. Specifically, while considering all subsets of $S_3$ as replacement sets is feasible, doing the same for subsets of $S_4$ is not, even when symmetry is taken into account.

We initiate the study of replacement sets containing patterns of length four, by focusing on cases which experimentally correspond with known number sequences from the On-Line Encyclopedia of Integer Sequences. We focus on patterns with no adjacency constraints since they seem more prone to yielding interesting number sequences, and we limit ourselves to sets of two patterns. Through computer computation, we are able to calculate the number of nontrivial equivalence classes in $S_n$ under $\Pi$-equivalence for all $n < 12$ and $\Pi$ consisting of two patterns of length four. We focus on the nontrivial equivalence classes due to the fact the number of trivial equivalence classes typically grows at an exponential rate and can be treated as a separate pattern-avoidance enumeration problem. For three sets of patterns, the computed number sequence appears to match an OEIS entry. Below, we study these three equivalences. For two of the equivalence relations, we are able to prove a formula for the number of classes, and for the third, we pose the formula as a conjecture.

Note that, due to symmetry, each of the three OEIS hits correspond with multiple pairs of patterns. For example the sequence resulting from the patterns $1234$ and $3421$ is the same as that resulting from $4321$ and $1243$. Now we explore the first of these $3$ sequences.

\begin{lemma}

For $n \geq 8$, all permutations in $S_n$ not beginning with $n$ that contain either of the patterns $1234$ or $3421$ are equivalent under the $\{1234,3421\}$-equivalence (with no adjacency constraints). Equivalently, these permutations are all equivalent to the identity.
\label{lemma12343421}
\end{lemma}

\begin{proof}

For $n=8$ this can be checked by computer. Now we will proceed by induction. Assume the proposition is true for $n-1$. Take any permutation $a = a_1 a_2 \cdots  a_n$ with $a_1 \neq n$ that contains at least one instance of one of the patterns. Then we claim $a$ is equivalent to the identity. Say our pattern is formed by the letters $a_{i_1}, a_{i_2}, a_{i_3}, a_{i_4}$. Choose some $j$ with $1 \leq j \leq n$ such that $j \neq 1, i_1,i_2,i_3,i_4, n$ and $a_j \neq n$. Then we can use our inductive hypothesis on all the $a_i$ except $a_j$ to order to rearrange them into increasing order. Now, since $j \neq 1$, the new first letter will be either one or two, and the first $n - 1$ letters will contain (several) $1234$ patterns, meaning we can apply the inductive hypothesis to them in order to rearrange them in increasing order. Since $j$ was not $n$, the new permutation will have its first $n-1$ letters in increasing order, and will end either with $n - 1$ or $n$. Applying the inductive hypothesis to the final $n - 1$ letters, we arrive at the identity, as desired.
\end{proof}

\begin{theorem}
		
For $n \geq 7$ the number of nontrivial equivalence classes in $S_n$ under the $\{1234,3421\}$ equivalence (with no adjacency constraints) is $n+28$.

\end{theorem}
	
\begin{proof}

We prove the theorem by induction. The $n=7$ case can be checked by computer. Assume the statement holds for $n-1$. Denote the equivalence classes for the permutations in $S_{n-1}$ by $A_1, A_2, \cdots, A_{n+27}$. Now for $i=1, 2, \cdots, n+27$ we can define $B_i=\{na_1a_2 \cdots a_{n-1} \mid a_1a_2 \cdots a_{n-1} \in A_i\}.$ We claim each of the $B_i$ is an equivalence class. Note that $n$ can never belong to either of the patterns $1234$ or $3421$ since it is the first letter. Therefore, the equivalence relation on the elements of $B_i$ can be seen as acting on only the final $n - 1$ letters, thereby making $B_i$ an equivalence class. By Proposition \ref{lemma12343421} all permutations not beginning with $n$ must belong to one equivalence class $C$ or belong to their own trivial equivalence class. Hence there are $n+28$ nontrivial classes.
\end{proof}

The second equivalence relation we study will have exponential growth in nontrivial equivalence classes, rather than linear growth. Nonetheless, we can use a very similar technique to analyze it.

\begin{lemma}

For $n \geq 8$ all permutations in $S_n$ not beginning or ending with $n$ that contain either of the patterns $1243$ or $3421$ are equivalent under the $\{1243,3421\}$-equivalence (with no adjacency constraints) to all other such permutations of the same parity. Notice that this means that each such permutation is thus equivalent to either $12n3(n-1)4567 \cdots (n-2)$ or $12n3(n-1)5467 \cdots (n-2)$.
\label{lemma12433421}
\end{lemma}

\begin{proof}

We can check that the result holds for $n=8$ by computer. We proceed by induction. Assume the proposition is true for $n-1$. Take any permutation $a_1 a_2 \cdots a_n$ with $a_1, a_n \neq n$ that contains at least one of the patterns. We will show that $a_1 \cdots a_n$ is equivalent to either $12n3(n-1)4567 \cdots (n-2)$ or $12n3(n-1)5467 \cdots (n-2)$. Call these $A$-form and $B$-form respectively. Assume our pattern uses the letters $a_{i_1}, a_{i_2}, a_{i_3}, a_{i_4}$. Pick a $j$ such that $1 \leq j \leq n$ and $j \neq i_1, i_2, i_3, i_4, n$ and $a_j \neq 1, n$. Then we can use our inductive hypothesis on all the $a_i$ except $a_j$ to order them into either $A$-form or $B$-form. By doing this, we guarantee that $n$ is either in the third or fourth position, and that there is pattern within the first five letters. Hence we can use the inductive hypothesis on the first $n-1$ letters, rearranging them to be either $A$-form or $B$-form. Notice that $n$ is now in the third position, and that the letters $2, 3, (n - 1), 4$ form a pattern within the final $n - 1$ letters. Therefore, we can apply the inductive hypothesis to the final $n-1$ letters. This places $n - 2$ in the final position (since the very first letter in the permutation will be $1$ or $2$). Moreover, since $n$ is now in the fourth position and the second through fourth letters form a pattern, we can apply the inductive hypothesis a final time to the first $n - 1$ letters, rearranging them to be either $A$-form or $B$-form. Combined with the fact that the final letter is $n - 2$, we see that the entire permutation is now either $A$-form or $B$-form.

We know that each non-avoiding permutation not beginning or ending with $n$ is equivalent to one of $12n3(n-1)4567 \cdots (n-2)$ or $12n3(n-1)5467 \cdots (n-2)$. Since the $\{1243, 3421\}$-equivalence preserves parity, it follows that such permutations are partitioned according to their parity.
\end{proof}

\begin{theorem}

For $n \geq 7$ the number of nontrivial equivalence classes of $S_n$ under the $\{1243,3421\}$ equivalence (with no adjacency constraints) is $7 \cdot 2^{n-4}-2$.
		
\end{theorem}
	
\begin{proof}

We will prove the theorem by induction. The base case of $n = 7$ can be checked by computer. Assume as an inductive hypothesis that the statement is true for $n-1$. Then say the equivalence classes for $S_{n-1}$ are $A_1, A_2, \cdots , A_{7 \cdot 2^{n-5}-2}$. Define $B_i=\{na_1a_2 \cdots a_{n-1} \mid a_1a_2 \cdots a_{n-1} \in A_i\}.$ and $C_i=\{a_1a_2 \cdots a_{n-1}n \mid a_1a_2 \cdots a_{n-1} \in A_i\}.$ Now since an $n$ in the beginning or end of a permutation can never be part of a pattern, it follows from the inductive hypothesis that each $B_i$ and each $C_i$ is a class in $S_n$. By Proposition \ref{lemma12433421}, since parity is an invariant, the remaining non-avoiding permutations fall into two nontrivial equivalence classes. Hence the total number of nontrivial equivalence classes is
$2 \cdot (7 \cdot 2^{n-5}-2) + 2 = 7 \cdot 2^{n-4}-2,$ 
completing the proof by induction.
\end{proof}

Now we look at the third and final OEIS sequence to appear. 

\begin{conjecture}

For $n \ge 7$ the number of nontrivial equivalence class of $S_n$ under the ${1234,3412}$-equivalence (with no adjacency constraints) is $\frac{n^3+6n^2-55n+54}{6}.$
\label{conjmain2}
\end{conjecture}

\noindent We suspect that the preceding conjecture can be proven with similar methods to the other results in this section. We have verified the conjecture for $n < 12$. Using the technique from the previous proofs in which we consider first the permutations not beginning with $n$, and then add in the other permutations via an induction, Conjecture \ref{conjmain2} reduces to the following.

\begin{conjecture}

For $n \ge 8$ the number of nontrivial equivalence classes dividing all permutations in $S_n$ not beginning with $n$ that contain either of patterns $1234$ or $3412$ is $\frac{n^2+3n-20}{2}$.

\end{conjecture}

\section{Conclusion}\label{secconclusion}

In this paper, we have presented a generalization of the Erd\"{o}s-Szekeres Theorem to permutation pattern-replacement equivalences, a characterization of equivalence classes under rotational equivalence, and the first results on enumerating nontrivial classes generated by pairs of patterns of length $4$. There are several directions of future work to consider. 
\begin{itemize}
    \item Our generalization of Erd\"{o}s-Szekeres holds for $n \ge 3k^2-4k+3$. We conjecture that the result can be tightened to $n \ge k^2-2k+3$. Moreover, when $k$ is even, this could potentially be reduced to $k^2 - 2k + 2$, matching the standard Erd\"{o}s-Szekeres Theorem.  
    \item We have proven that under rotational equivalence there is one nontrivial equivalence class when $n \ge 2c - 1$. We conjecture that this bound is tight exactly when $m$ is an alternating pattern.
    \item We have enumerated the nontrivial equivalence classes as OEIS sequences for two relations given by pairs of patterns of length $4$. A direction of future work is to prove the third OEIS match (Conjecture \ref{conjmain2}) and to find additional sets of patterns which yield interesting number sequences. These could arise from changing adjacency constraints or mixing patterns of lengths $3$ and $4$.
    \item In our work, we have studied two infinite families of replacement equivalences, the rotational-equivalence with adjacency constraints, and the Erd\"{o}s-Szekeres equivalence. Additionally, \cite{kuszmaul2014new} and \cite{kuszmaul2013equivalence} have studied several other infinite families. Continuing to identify infinite families of particular interest is an important direction of future work, since there are far too many sets of patterns of small size to consider each individually.%
\end{itemize}

\section{Acknowledgements}\label{secacknow}

I would like to thank my mentor William Kuszmaul for helping and advising me in this research, and the MIT PRIMES Program for giving me the opportunity to perform the research.

\newpage
\bibliographystyle{plain}
\bibliography{main}

\end{document}